\documentclass[reqno,12pt,centertags]{amsart}
\usepackage{geometry}
\usepackage{amsmath,amsthm,amscd,amssymb,latexsym,upref,stmaryrd,enumitem}
\usepackage{graphicx}
\usepackage{hyperref}
\newcommand*{\mailto}[1]{\href{mailto:#1}{\nolinkurl{#1}}}
\usepackage[numbers,sort&compress]{natbib}

\newcommand{\beq}{\begin{equation}}
\newcommand{\eeq}{\end{equation}}
\newcommand{\ba}{\begin{align}}
\newcommand{\ea}{\end{align}}

\renewcommand{\ln}{\text{\rm ln}}

 \allowdisplaybreaks
\numberwithin{equation}{section}

\newtheorem{theorem}{Theorem}[section]
\newtheorem{lemma}[theorem]{Lemma}
\newtheorem{corollary}[theorem]{Corollary}

\theoremstyle{definition}

\newtheorem{remark}{Remark}[section]

\begin{document}


\title[Inverse problems for perturbed Bessel operator]
{The inverse eigenvalue problems for perturbed Bessel operator with mixed data}

\author[Z.~G.~Liu]{Zeguang Liu}
\address{School of Mathematics and Statistics, Northeast (Dongbei) Normal University, Changchun, 130024, Jilin, People's Republic of China}
\email{\mailto{liuzeguang205@nenu.edu.cn}}

\author[X.~J.~Xu]{Xin-Jian Xu}
\address{Department of Mathematics, School of Mathematics and Statistics, Nanjing University of
Science and Technology, Nanjing, 210094, Jiangsu, People's
Republic of China}
\email{\mailto{xxjxuxinjian@163.com}}

\date{\today}

\begin{abstract}
{We consider inverse eigenvalue problems for the perturbed Bessel operator in $L^{2}(0,1)$. (1) For the case where the angular-momentum quantum number $\ell\in\mathbb{N}\cup\{0\}$, we establish a uniqueness result for the inverse spectral problem by utilizing the closedness condition of a certain function system constructed based on the eigenvalues and the norming constants. (2) For the broader case where $\ell \geq -1/2$, we provide a uniqueness result for the inverse problem by using the density condition satisfied by the eigenvalues and the norming constants, where an additional smoothness condition may be imposed on the potential. (3) In the last section of this article, we present some corollaries based on (2). The results in these corollaries have already been established for the case $\ell=0$ by Gesztesy, Simon, Wei, Xu, Hatino\v{g}lu, et al., and we extend these results to the general case $\ell \geq -1/2$.}
\end{abstract}

\maketitle

\section{Introduction}

In this paper, we study the boundary value problem on the interval $(0,1)$:
\begin{align}\label{e1.1}
L(\ell, q)(f):=-f^{\prime \prime}(x)+\frac{\ell(\ell+1)}{x^{2}}f(x)+q(x)f(x)
=\lambda f(x),
\end{align}
subject to the boundary conditions
\begin{align}\label{e1.2}
\lim_{x \rightarrow 0}\frac{f(x)}{x^{\ell+1}}<\infty, ~~f^{\prime}(1)+\beta f(1)=0,
\end{align}
where $\lambda$ is the spectral parameter, $\ell \geq-1/2$ is the angular-momentum quantum number, 
$q$ is a real-valued potential function in $L^{1}(0,1)$, 
and $L(\ell, q)$ is the self-adjoint operator generated by $-D^2+\ell(\ell+1)x^{-2}+q$ in $L^{2}(0,1)$ with the boundary condition \eqref{e1.2}.
The parameter $\beta$ belongs to $\mathbb{R} \cup\{\infty\}$;  
if $\beta=\infty$, then the boundary condition $f^{\prime}(1)+\beta f(1)=0$ means the Dirichlet boundary condition $f(1)=0$. 

Given $q\in L^{1}(0, 1)$, the operator $L(\ell, q)$ has a simple discrete spectrum denoted by $\sigma(\ell, q, \beta)$, and the eigenvalue $\lambda_{\ell, \beta, n}(q)\in\sigma(\ell, q, \beta)$ has the following asymptotic behavior:
\begin{equation}\label{e1.3}
\sqrt{\lambda_{\ell, \infty, n}(q)}=\left(n+\frac{\ell}{2}\right) \pi+O(\textnormal{R}(n^{2})+n^{-1}), ~~n \in \mathbb{N},
\end{equation}
\begin{equation}\label{e1.4}
\sqrt{\lambda_{\ell, \beta, n}(q)}=\left(n+\frac{\ell-1}{2}\right) \pi+O(\textnormal{R}(n^{2})+n^{-1}), ~~\beta \in \mathbb{R}, 
~~n \in \mathbb{N},
\end{equation}
where
\begin{equation}\label{e1.7}
\textnormal{R}(\lambda):=\int_0^1 
\frac{y\tilde{q}(y)}{1+|\lambda|^{1/2}y}dy, \;\text{and}\; 
\tilde{q}:=
\begin{cases}
|q|,&\textnormal{if}~~\ell>-1/2,\\
|(1-\ln(x))q|,&\textnormal{if}~~\ell=-1/2,
\end{cases}
\end{equation}
see \cite[Theorem 2.5]{KST2010}.

Let $\psi_{\ell}(\lambda, x, q)$ be a solution of the equation \eqref{e1.1} under the initial conditions
\begin{align}\label{e1.8}
\begin{cases}
\psi_{\ell}(\lambda, 1, q)=1, \psi_{\ell}^{\prime}(\lambda, 1, q)
=-\beta, 
&\textnormal{if}~\beta \in \mathbb{R}, \\
\psi_{\ell}(\lambda, 1, q)
=0, \psi_{\ell}^{\prime}(\lambda, 1, q)=-1, 
&\textnormal{if}~\beta=\infty.
\end{cases}
\end{align}
For every fixed $x\in(0,1]$, the function $\psi_{\ell}(\lambda, x, q)$ is an entire function of $\lambda$. 
The norming constant corresponding to the eigenvalue $\lambda_{\ell, \beta, n}$ is given by
\begin{align*}
\zeta(\lambda_{\ell, \beta, n},\beta, q)=\int_0^1|\psi_{\ell}(\lambda_{\ell, \beta, n}, x, q)|^{2}dx.
\end{align*}

The perturbed Bessel operators $L(\ell, q)$ with integer or half-integer $\ell$ originate from the Schr\"{o}dinger operators $-\Delta+q$ acting in $L^2(\mathbb{R}^n)$ ($n\in\mathbb{N}$) with spherically symmetric potentials (see \cite[p.286]{Wei1987}). For general $\ell\geq-1/2$, the operators $L(\ell, q)$ appear, for example, in the study of zonal Schr\"{o}dinger operators on spheres and Laplace operators for a Riemannian manifold that is a hypersurface of revolution \cite{ Carl1997}. Inverse problems for the perturbed Bessel operators $L(\ell, q)$ on a finite interval attract the attention of many researchers. It is well known that the spectral measure (or the eigenvalues and the corresponding norming constants, or the Weyl-Titchmarsh $m$-function) can uniquely determine the potential function. Furthermore, the two sets of eigenvalues corresponding to the perturbed Bessel operators under different boundary conditions can also uniquely determine the potential function. 
These inverse problems are studied, for example, in \cite{Borg1946,Guli2001,Levit1987,Marc1986,PT1987} for the case $\ell = 0$, and in \cite{SRYal2007, Carl1997,Eckh2014,Guli2005,Gasy1965,HS2010,KST2010,Ser2007,LAVS1994} for more general values of $\ell$.

Subsequently, in order to establish the uniqueness of the inverse problem, many studies have been conducted on inverse problems in which other types of data are considered. The given data include the following aspects:
\begin{enumerate}[label=\textnormal{(\arabic*).}]
\item A set of eigenvalues possibly taken from infinitely many different spectra.
\item A set of norming constants corresponding to known eigenvalues.
\item The potential itself on the interval $(a, 1) \subset(0, 1)$.
\item The smoothness of the potential in the neighbourhood of $a$.
\end{enumerate}

For the perturbed Bessel operator $L(\ell, q)$ with $\ell=0$, the inverse problems based on the mixed data sets (1)--(4) are extensively studied; see, e.g., \cite{AFR2009,del1997,GS2000,Hat2021,HL1978,Horv2005,HorvO2016,WX2012}, and the references therein.

For the perturbed Bessel operators $L(\ell, q)$ with $\ell$ being an integer or a real number, there are also some studies on inverse problems based on mixed data from (1)--(4). The inverse problem of recovering the potential function $q$ from the data in (1) is discussed in \cite{Koy2009}. The inverse problem with the data from (3) is the so-called half-inverse problem, which can be seen in \cite{ET2013,Guli2023,HE2005}, where one spectrum is used. The inverse problem with the mixed data from (1) and (3) is discussed in \cite{XYB2023}. The inverse problem based on the mixed data from (2) and (3) with one spectrum is discussed in \cite{YGJ2019}. 

There are also some studies on using other types of spectral data to solve the inverse spectral problem for the operator $L(\ell, q)$. For instance, the problem of recovering the potential function $q$ using the spectra of $L(\ell, q)$ corresponding to different values of $\ell$ (see, e.g., \cite{Aceto2008,CS1994,Chr1993,WPRun2001}),  and the so-called three-spectra inverse problem (see, e.g., \cite{XY2024}). Moreover, the inverse problem theory also exhibits a rich body of research results for other differential operators with Bessel-type singularities (see, e.g., \cite{SRYalA2007,Bondarenko2018,OYurko2002,Yurko2002}). Research on inverse spectral problems over infinite intervals is also extensive; however, it is less relevant to the present paper and thus will not be discussed here.

In this paper, we aim to recover the potential $q$ from the mixed data (1)--(3) and the mixed data (1)--(4), respectively. The main results of this paper are Theorems \ref{t4.1}--\ref{t4.3}, where the mixed data (1)--(3) are used in Theorems \ref{t4.1} and \ref{t4.2}, and the mixed data (1)--(4) are used in Theorem \ref{t4.3}. For $\ell=0$, similar results can be found in \cite{HorvO2016}.

Our criterion for uniqueness based on the mixed data (1)--(3) is related to the closedness of a certain function system, which is constructed using eigenvalues and the corresponding norming constants. To formulate the results, we need the following definition, which is introduced in \cite{Horv2005}. For $p \in(1, \infty), 1/p+1/p^{\prime}=1$ and $a>0$, a system $\{\varphi_n \mid \varphi_n \in L^{p^{\prime}}(0, a), n \geq 1\}$ is called \textbf{closed} in $L^{p}(0, a)$ if $h \in L^{p}(0, a)$ and $\textstyle\int_0^a h(x) \varphi_n(x)dx=0$ for all $n$ implies $h=0$. It is easy to verify that the closedness of $\{\varphi_n \mid  n \geq 1\}$ in $L^{p}(0, a)$ for $p \in(1, \infty)$ is equivalent to the completeness of $\{\varphi_n \mid  n \geq 1\}$ in $L^{p^{\prime}}(0, a)$.

For a set $\Lambda=\{\lambda_{1}, \lambda_{2}, \lambda_{3}, \cdots\} \subset \mathbb{R}$ and a subset $S \subset \Lambda$ we define the function system:
\begin{align*}
C(\Lambda, S)
:=\{\cos(2\sqrt{\lambda_n}x)\mid\lambda_n\in\Lambda\} 
\cup\{x\sin(2\sqrt{\lambda_n}x)\mid\lambda_n\in S\}.
\end{align*}
Given the mixed data (1)--(3), we can formulate the following theorem.

\begin{theorem}\label{t4.1}
Let $\ell\in \mathbb{N}\cup\{0\}$, $p\in(1, \infty)$, $q\in L^p(0,1)$, and $a \in(0,1]$. 
Suppose the elements $\lambda_n$ in
$\Lambda
=\{\lambda_{n}\mid\lambda_{n}\in\sigma(\ell, q, \beta_n), 
n\in\mathbb{N}\}$ 
satisfy $\lambda_n \nrightarrow -\infty$, and let $S \subset \Lambda$ be a subset such that the values $\zeta(\lambda_n, \beta_n, q)$ are known a priori for all $\lambda_n \in S$.
If the system
\begin{align*}
S_{\ell}(\Lambda,S)
:= 
\begin{cases}
\{x^{2k}\mid k=0,1,2,\cdots,\ell\}\cup C(\Lambda, S), 
& 0 \notin \Lambda, \\ 
\{x^{2k}\mid k=0,1,2,\cdots,\ell+1\}\cup C(\Lambda, S), 
& 0 \in \Lambda, 0\notin S, \\ 
\{x^{2k}\mid k=0,1,2,\cdots,\ell+2\}\cup C(\Lambda, S), 
& 0 \in S,
\end{cases}
\end{align*}
is closed in $L^p(0, a)$, then $q$ on $(a, 1), \Lambda$, and $S$ uniquely determine $q \in L^p(0, a)$.
\end{theorem}

Note that in this paper, we use the term “uniquely determined $q$” to mean that there is no other potential $\hat{q} \neq q$ that satisfies the same conditions in the theorems or corollaries.

Moreover, we show that, under an additional assumption, the condition in Theorem \ref{t4.1} is not only sufficient but also necessary for uniqueness. For this, we introduce the operator $A_{q}$.
For $q$, $\hat{q}\in L^{p}(0,a)$, $p\in(0,\infty)$ and $a\in(0,1]$,
the linear operator $A_{\hat{q}}: L^p(0, a) \rightarrow L^p(0, a)$ is continuous and defined as
\begin{align}\label{l3.8}
A_{\hat{q}}(f)(x)
=f(x)+2 \int_{x}^{a} f(t) \frac{1}{t} M_{\ell}(t, x, q, \hat{q})dt,  
~~f\in L^p(0, a),
\end{align}
where the function $M_{\ell}(t, x, q, \hat{q})$ is continuous and uniformly bounded in the region 
$\{(x, t) \in \mathbb{R}^2 \mid 0 \leq t \leq x \leq 1\} \setminus\{(0,0)\}$. 
In addition,
\begin{align*}
A_{q}(f)(x)
=f(x)+2 \int_{x}^{a} f(t) \frac{1}{t} M_{\ell}(t, x, q, q)dt,  
~~f\in L^p(0, a).
\end{align*}
For further details on $M_{\ell}$, we refer the reader to \cite{XYB2023}. This article, however, does not delve into its specific properties and thus omits an in-depth introduction.

\begin{theorem}\label{t4.2}
Suppose $\ell\in \mathbb{N}\cup\{0\}$.
If the system $S_{\ell}(\Lambda,S)$ is not closed in $L^p(0, a)$, $p\in(1,\infty)$, and the continuous linear operator $A_q$ is an isomorphism, then $q$ on $(a,1)$, $\Lambda$, and $S$ do not uniquely determine $q \in L^p(0, a)$.
\end{theorem}

We can replace the knowledge of finitely many eigenvalues (and norming constants) with the knowledge of the derivatives of $q$ at $a$. Let us define the common counting function of the eigenvalues and norming constants by
\begin{align*}
m(t):=2n_{\Lambda}(t^2)+2n_{S}(t^2),
\end{align*}
where the counting function is defined as
\begin{align*}
n_{A}(t)
=\sum_{\substack{\lambda\in A \\ |\lambda|\leq t}}1,
\end{align*}
for a set $A$ of certain numbers. 
For $1\leq p<\infty$, denote by $p^{\prime}$ the conjugate exponent, i.e., $1/p+1/p^{\prime}=1$ for $p\in(1,\infty)$, just as we have used earlier in this article; moreover, if $p=\infty$, then $p^{\prime}=1$.
Next, we present a sufficient condition for the uniqueness of the following inverse problem.

\begin{theorem}\label{t4.3}
Let $\ell\geq-1/2$, $\hat{q}$, $q\in L^{1}(0,1)$, and $a \in(0,1]$.
Suppose that $\hat{q}=q$ a.e. on $(a,1)$. Given some distinct eigenvalues $\lambda_n\in \sigma(\ell,q,\beta_{n})\cap\sigma(\ell,\hat{q},\beta_{n}), n\in\mathbb{N}$, $\lambda_{n}\nrightarrow -\infty$, and $\zeta(\lambda_{n},\beta_{n}, q)=\zeta(\lambda_{n},\beta_{n}, \hat{q})$, $\lambda_{n}\in S$, for some $S \subset \Lambda$.

For $1\leq p<\infty$, assume further $q-\hat{q} \in W^{k, p}((a-\delta_0, a])$ and $(q-\hat{q})^{(i)}(a)=0$, for some $\delta_0>0$, $k\in\mathbb{N}\cup\{0\}$, $i=0,1,\cdots, k-1$.  
If there exists a sequence $R_{i}\rightarrow\infty$ such that
\begin{align}\label{e1.0}
\limsup_{i \rightarrow \infty}
\left[\int_{0}^{R_{i}}\frac{m(t)}{t}dt
-\frac{4aR_{i}}{\pi}+\left(k+2\ell+2
+\frac{1}{p^{\prime}}\right) \ln R_{i}\right]
>-\infty.
\end{align}
Then $q=\hat{q}$ a.e. on $(0,a)$.

For $p=\infty$, assume further $q-\hat{q} \in C^{k}((a-\delta_0, a])$ and $(q-\hat{q})^{(i)}(a)=0$ for $i=0, \cdots, k$. Then \eqref{e1.0} implies the same conclusion.
\end{theorem}

The paper is organized as follows. 
In Sections \ref{s2} and \ref{s3}, we introduce two auxiliary functions: the characteristic function $\Delta(\lambda, q)$ and the entire function $H_{\ell}(\lambda)$, respectively. 
In Section \ref{s4}, the proofs of Theorems \ref{t4.1}--\ref{t4.3} are given, and the methods and techniques we used here are mainly based on the works \cite{Horv2005,HorvO2016,XYB2023}. 
In Section \ref{s5}, we present some corollaries of Theorem \ref{t4.3}, where some ideas are adopted from \cite{Hat2021}.
The main results of this paper are Theorems \ref{t4.1}--\ref{t4.3}; moreover, Corollaries \ref{t5.3}, \ref{t5.2}, \ref{t5.6}, \ref{t5.7}, \ref{t5.9}, and \ref{t5.8} also have their interest, and some of these corollaries improve the previous results even in the case $\ell=0$.

\section{The characteristic function $\Delta(\lambda,q)$}\label{s2}

Recall apart from the solution $\psi_{\ell}(\lambda, x, q)$ defined in \eqref{e1.8}, there is another solution $\phi_{\ell}(\lambda, x, q)$ for the equation \eqref{e1.1} that satisfies 
\begin{align*}
\lim_{x \rightarrow 0} x^{-\ell-1} \phi_{\ell}(\lambda, x, q)=\frac{\sqrt{\pi}}{\Gamma\left(\ell+\frac{3}{2}\right)2^{\ell+1}},
\end{align*}
where $\Gamma(\cdot)$ denotes the gamma function, see \cite{KST2010}. 
From \cite[Lemma 2.2]{KST2010} and \cite[(9.2.5) and (9.2.11)]{Abra1972}, it follows that for every fixed $x\in (0,1]$, the function $\phi_{\ell}(\lambda, x, q)$ is an entire function of $\lambda$, and it has the asymptotic formulas as $|\lambda| \rightarrow \infty$: 
\begin{align}
\phi_{\ell}(\lambda, x, q)
&=\lambda^{-\frac{\ell+1}{2}}
\left(\sin \left(\sqrt{\lambda} x-\frac{\ell \pi}{2}\right)+O\left((\textnormal{R}(\lambda)+|\lambda|^{-\frac{1}{2}}) e^{x| \textnormal{Im}(\sqrt{\lambda}) \mid}\right)\right),\label{e1.5}\\
\phi_{\ell}^{\prime}(\lambda, x, q)
&=\lambda^{-\frac{\ell}{2}}
\left(\cos \left(\sqrt{\lambda} x-\frac{\ell \pi}{2}\right)+O\left((\textnormal{R}(\lambda)+|\lambda|^{-\frac{1}{2}}) e^{x|\textnormal{Im}(\sqrt{\lambda})|}\right)\right),\label{e1.6}
\end{align}
uniformly in $x\in(x_0,1)$ for every fixed $0<x_0<1$, where $\textnormal{R}(\lambda)$ is defined in \eqref{e1.7}. 

Denote
\begin{align*}
\Delta(\lambda, q)
:=W(\psi_{\ell}(\lambda, x, q), \phi_{\ell}(\lambda, x, q)).
\end{align*}
The function $\Delta(\lambda, q)$ is called the characteristic function.
Clearly,
\begin{align*}
\Delta(\lambda, q)
&=
\begin{cases}
\phi_{\ell}^{\prime}(\lambda, 1, q)+\beta\phi_{\ell}(\lambda, 1, q), 
&\textnormal{if}~\beta\in\mathbb{R}, \\
\phi_{\ell}(\lambda, 1, q), 
&\textnormal{if}~\beta=\infty,
\end{cases}
\end{align*}
thus the function $\Delta(\lambda, q)$ is an entire function of $\lambda$. The zeros $\{\lambda_n\}$ of the characteristic function $\Delta(\lambda, q)$ coincide with the eigenvalues of the boundary value problem \eqref{e1.1} and \eqref{e1.2}. The functions $\phi_{\ell}(\lambda_n, x, q)$ and $\psi_{\ell}(\lambda_n, x, q)$ for $\lambda_n\in \sigma(\ell, q, \beta)$ are eigenfunctions, and there exists a sequence $\{\kappa(\lambda_n, \beta, q)\}$ such that
\begin{align}\label{V1}
\psi_{\ell}(\lambda_n, x, q)
=\kappa(\lambda_n, \beta, q)\phi_{\ell}(\lambda_n, x, q), ~~\kappa(\lambda_n, \beta, q) \neq 0.
\end{align}

We use the notation 
\begin{align*}
\tau(\lambda, q):=\int_0^1|\phi_{\ell}(\lambda, x, q)|^{2}dx.
\end{align*}

\begin{lemma}\label{l2.2}
Suppose  $\lambda_n\in\sigma(\ell, q, \beta)$. Then the following relation holds,
\begin{align*}
\dot{\Delta}(\lambda_n, q):=\frac{d\Delta(\lambda, q)}{d\lambda}\bigg|_{\lambda=\lambda_n}=
-\tau(\lambda_n, q)\kappa(\lambda_n, \beta, q).
\end{align*}
\end{lemma}

The proof of Lemma \ref{l2.2} is established in \cite[Lemma 1.1.1, p.7]{Guli2001} under the conditions $\ell=0$ and $q\in L^{2}(0,1)$, rather than $\ell\geq-1/2$ and $q\in L^{1}(0,1)$. The proof is quite similar and omitted.

\section{The entire function $H_{\ell}(\lambda)$}\label{s3}

Let $q$, $\hat{q} \in L^{1}(0, a)$, we use the notation
\begin{align*}
H_{\ell}(\lambda)
:=\left|\begin{array}{ll}
\phi_{\ell}(\lambda, 1, q) 
& \phi_{\ell}(\lambda, 1, \hat{q}) \\
\phi_{\ell}^{\prime}(\lambda, 1, q) 
& \phi_{\ell}^{\prime}(\lambda, 1, \hat{q})
\end{array}\right|.
\end{align*}
Then the function $H_{\ell}(\lambda)$ has the representation (see \cite[Lemma 3.2]{XYB2023})
\begin{align*}
H_{\ell}(\lambda)
=\int_0^1(q(x)-\hat{q}(x))\phi_{\ell}(\lambda, x, q) \phi_{\ell}(\lambda, x, \hat{q})dx.  
\end{align*}
Obviously, if $q(x)=\hat{q}(x)$ a.e. on $(a, 1)$, one has 
\begin{align}\label{e3.1}
H_{\ell}(\lambda)
=\int_0^a(q(x)-\hat{q}(x))\phi_{\ell}(\lambda, x, q)\phi_{\ell}(\lambda, x, \hat{q})dx. 
\end{align}

Next, we give some properties of the entire function $H_{\ell}(\lambda)$.

\begin{lemma}\cite[(4.3)]{XYB2023}\label{l3.1}
Suppose $\ell\in \mathbb{N}\cup\{0\}$, $q$, $\hat{q}\in L^p(0,1)$, $p\in(1,\infty)$, and $q(x)=\hat{q}(x)$ a.e. on $(a, 1)$. Then
\begin{equation}\label{l3.7}
\begin{split}
\lambda^{\ell+1} H_{\ell}(\lambda)
=&\frac{1}{2}\int_0^a(q(x)-\hat{q}(x))dx \\
&+\frac{1}{2}\int_0^a \cos(2\sqrt{\lambda}x)(A_{\hat{q}}((q-\hat{q})(x)) dx, 
\end{split}
\end{equation}
where the operator $A_{\hat{q}}$ is defined in \eqref{l3.8}.
\end{lemma}

\begin{remark}\label{r3.2}
Suppose $\ell\in \mathbb{N}\cup\{0\}$ and $q$, $\hat{q}\in L^{p}(0, a)$, $p>1$.
If $\lambda_n \in \sigma(\ell, q, \beta_{n}) \cap \sigma(\ell, \hat{q}, \beta_{n})$ for $n \in\mathbb{N}$ are pairwise distinct and $\lambda_n \nrightarrow-\infty$, then
\begin{align}\label{e3.2}
\int_{0}^{a}(q(x)-\hat{q}(x))dx=0.
\end{align}
A detailed verification of this process can be found in \cite[proof of Theorem 1.2]{XYB2023}.
\end{remark}

\begin{remark}\label{r3.3}
We observe that for all $c\in(0,a)$,
$A_{\hat{q}}(f)(x)= 0$
is the Volterra integral equation on the interval $(c,a)$. Therefore, if $A_{\hat{q}}(f)(x)= 0$ a.e. on $(c,a)$, then $f=0$ a.e. on $(c,a)$. From the arbitrariness of $c$, it follows that if $A_{\hat{q}}(f)(x)= 0$ a.e. on $(0,a)$, then $f=0$ a.e. on $(0,a)$.  
\end{remark}

\begin{lemma}\label{l3.4}
Suppose $\ell\geq-1/2$ and $q$, $\hat{q}\in L^{1}(0, a)$. Then it holds:  
\begin{enumerate}[label=\textnormal{(\arabic*).}]
\item $H_{\ell}(\lambda)=0$ if and only if there exists $\beta \in \mathbb{R} \cup\{\infty\}$ such that
\begin{align*}
\lambda\in\sigma(\ell,q,\beta)\cap\sigma(\ell,\hat{q},\beta).
\end{align*}
\item If $\lambda \in \sigma(\ell, q, \beta) \cap \sigma(\ell, \hat{q}, \beta)$, then  $\zeta(\lambda,\beta, \hat{q})=\zeta(\lambda,\beta, q)$ if and only if 
\begin{align*}
\dot{H}_{\ell}(\lambda):=\frac{\partial H_{\ell}(\lambda)}{\partial \lambda}=0.
\end{align*}
\end{enumerate}
\end{lemma}

\begin{proof}
(1). The proof of the case where $\ell\in\mathbb{N}\cup\{0\}$ can be found in \cite[Lemma 3.1]{XYB2023}; the more general case $\ell\geq-1/2$ can be proved similarly.

(2). If $\beta=\infty$, $\lambda \in \sigma(\ell, q, \infty) \cap \sigma(\ell, \hat{q}, \infty)$, then
\begin{equation}\label{l3.9}
\begin{split}
\dot{H}_{\ell}(\lambda)
&=\frac{\partial(\phi_{\ell}(\cdot,1,q) \phi_{\ell}^{\prime}(\cdot,1,\hat{q})
-\phi_{\ell}(\cdot,1,\hat{q})\phi_{\ell}^{\prime}(\cdot,1,q))}
{\partial \lambda}(\lambda) \\
&=\phi_{\ell}^{\prime}(\cdot, 1, \hat{q}) 
\frac{\partial\phi_{\ell}(\cdot,1,q)}{\partial\lambda}(\lambda)
-\phi_{\ell}^{\prime}(\cdot, 1, q) 
\frac{\partial\phi_{\ell}(\cdot,1,\hat{q})}{\partial\lambda}(\lambda) \\
&=\phi_{\ell}^{\prime}(\lambda, 1, \hat{q})
\dot{\Delta}(\lambda, q)
-\phi_{\ell}^{\prime}(\lambda, 1, q)
\dot{\Delta}(\lambda, \hat{q}) \\
&=-\kappa(\lambda,\infty,q)\tau(\lambda,q)\phi_{\ell}^{\prime}(\lambda,1,\hat{q})
+\kappa(\lambda,\infty,\hat{q})\tau(\lambda,\hat{q})\phi_{\ell}^{\prime}(\lambda,1,q)\\
&=\frac{\kappa^2(\lambda,\infty,\hat{q})\tau(\lambda,\hat{q})
-\kappa^2(\lambda,\infty,q)\tau(\lambda,q)}
{\kappa(\lambda,\infty,q)\kappa(\lambda,\infty,\hat{q})},
\end{split}
\end{equation}
the fourth equality in \eqref{l3.9} holds by Lemma \ref{l2.2}, and the last equality in \eqref{l3.9} holds by the identity \eqref{V1}. 
From \eqref{V1}, we observe that
\begin{equation}\label{l3.10}
\begin{split}
\kappa^2(\lambda,\infty,q)\tau(\lambda,q)
=&\int_{0}^{1}|\kappa(\lambda,\infty,q)\phi_{\ell}(\lambda,x,q)|^{2}dx\\
=&\int_{0}^{1}|\psi_{\ell}(\lambda, x, q)|^{2}dx=\zeta(\lambda,\infty, q).
\end{split}
\end{equation}
From \eqref{l3.9} and \eqref{l3.10},  it follows that  $\dot{H}_{\ell}(\lambda)=0$ if and only if 
$\zeta(\lambda,\infty, \hat{q})=\zeta(\lambda,\infty,q)$.

The proof of the case $\beta \in \mathbb{R}$ is similar to the case $\beta=\infty$ and is omitted.
\end{proof}

\begin{lemma}\label{l3.6}
Suppose $\ell\geq-1/2$, $q$, $\hat{q}\in L^{1}(0, a)$ satisfying $q=\hat{q}$ a.e. on $(a, 1)$, for some $0<a\leq1$. Assume further: 

For $1\leq p<\infty$ and $k\in\mathbb{N}\cup\{0\}$, there exists $\delta_0>0$ such that $q-\hat{q}\in W^{k,p}((a-\delta_0, a])$, and
$(q-\hat{q})^{(i)}(a)=0$, for $i=0,1, \cdots, k-1$.
Then for any $\varepsilon>0$, there exists $\delta(\varepsilon)>0$ such that
\begin{equation}\label{e3.7}
|H_{\ell}(z^2)| 
\leq \frac{e^{2a|\textnormal{Im}z|}}
{|z|^{2\ell+2}|\textnormal{Im}z|^{k+\frac{1}{p^{\prime}}}}
\left(\varepsilon
+C|z|^{\gamma}e^{-\delta(\varepsilon)|\textnormal{Im}z|}
+C|\textnormal{Im}z|^{k+\frac{1}{p^{\prime}}}
e^{-2\delta(\varepsilon)|\textnormal{Im}z|}\right), 
\end{equation}
where $\textnormal{Im}z \neq 0$,
and the constants $C>0$ and $0<\gamma<1$ are independent of $\varepsilon, \delta$ and $z$.

For $p=\infty$, if $q-\hat{q} \in C^{k}((a-\delta_0, a])$ and $(q-\hat{q})^{(i)}(a)=0$ for $i=0, \cdots, k$, then \eqref{e3.7} holds.
\end{lemma}

\begin{proof}
Fix $x_0\in(0, a-\delta_0)$.
By the asymptotic estimates \eqref{e1.5} and \eqref{e1.6}, we can choose a constant $0<\gamma<1$ such that 
\begin{equation}\label{l3.11}
\phi_{\ell}(z^2, x, q)
=z^{-\ell-1}\left(\sin \left(z x-\frac{\ell\pi}{2}\right)
+O(|z|^{\gamma-1}e^{x|\textnormal{Im}z|})\right),
\end{equation}
and
\begin{equation}\label{l3.12}
\phi_{\ell}^{\prime}(z^2, x, q)
=z^{-\ell}
\left(\cos \left(zx-\frac{\ell \pi}{2}\right)+O\left(|z|^{\gamma-1} e^{x|\textnormal{Im}(z)|}\right)\right),
\end{equation}
uniformly in $x\in[x_0,1]$. The identity \eqref{e3.1} can be rewritten as
\begin{equation}\label{l3.13}
\begin{split}
H_{\ell}(z^2)=&\left(\int_{0}^{x_0}+\int_{x_0}^{a}\right)(q-\hat{q})(x)\phi_{\ell}(\lambda, x, q) \phi_{\ell}(\lambda, x, \hat{q})dx\\
=&\phi_{\ell}(z^2,x_0,q)\phi_{\ell}^{\prime}(z^2,x_0,\hat{q})-\phi_{\ell}(z^2,x_0,\hat{q})\phi_{\ell}^{\prime}(z^2,x_0,q)\\
&+\int_{x_0}^{a}(q-\hat{q})(x)\phi_{\ell}(\lambda, x, q) \phi_{\ell}(\lambda, x, \hat{q})dx.
\end{split}
\end{equation}
Putting the estimates \eqref{l3.11} and \eqref{l3.12} into the identity \eqref{l3.13}, for $\delta\in(0,\delta_{0})$, one has
\begin{align}\label{e3.3}
\begin{split}
H_{\ell}(z^2)= &z^{-2 \ell-2}\int_{x_0}^{a}(q-\hat{q})(x)f_0(x)dx+z^{-2\ell-2}g(z)\\
= &z^{-2 \ell-2}\left(\int_{x_0}^{a-\delta}+\int_{a-\delta}^{a}\right)
(q-\hat{q})(x)f_0(x)dx+z^{-2\ell-2}g(z),
\end{split}
\end{align}
where $f_0(x)=O\left(e^{2x|\textnormal{Im}z|}\right)$ uniformly in $x\in[x_0,a]$, and $g(z)=O\left(|z|^{\gamma}e^{2x_0|\textnormal{Im}z|}\right)$, 
for $|z|\rightarrow\infty$. 

On the one hand, 
since $q-\hat{q}\in W^{k, p}(a-\delta_0, a]$ and $(q-\hat{q})^{(i)}(a)=0$ for $i=0, \cdots, k-1$, and $1\leq p<\infty$, 
by performing $k$ iterations of integration by parts, we have 
\begin{equation*}
\begin{split}
\int_{a-\delta}^a(q-\hat{q})(x)f_{0}(x)dx=&(-1)\int_{a-\delta}^a(q-\hat{q})^{\prime}(x)f_{1}(x)dx=\cdots\\
=&(-1)^{k}\int_{a-\delta}^a(q-\hat{q})^{(k)}(x)f_{k}(x)dx,
\end{split}
\end{equation*}
where 
\begin{align*}
f_{i+1}(x)=\int_{a-\delta}^{x} f_i(t)dt.
\end{align*}
We see by induction on $i$ that
\begin{equation}\label{l3.14}
f_i(x)
=O\left(\frac{e^{2|\textnormal{Im}z|x}}
{|\textnormal{Im}z|^{i}}\right)
\end{equation}
uniformly in $x\in[a-\delta,a]$. Note that H\"{o}lder's inequality gives 
\begin{equation}\label{l3.15}
\int_{a-\delta}^{a}
\left|(q-\hat{q})^{(k)}(x)\right|e^{2|\textnormal{Im}z|x}dx \leq C_1\|(q-\hat{q})^{(k)}\|_{L^{p}(a-\delta, a)}
\frac{e^{2|\textnormal{Im}z|a}}{|\textnormal{Im}z|^{1/p^{\prime}}},
\end{equation}
for some constant $C_1>0$. For small $\delta$, the $L^{p}$-norm of $(q-\hat{q})^{(k)}$ is small. 
The facts \eqref{l3.14} and \eqref{l3.15} show that for small $\delta$, we have 
\begin{equation}\label{e3.4}
\left|\int_{a-\delta}^{a}(q-\hat{q})(x)f_{0}(x)dx\right| 
\leq \varepsilon \frac{e^{2|\textnormal{Im}z|a}}{|\textnormal{Im}z|^{k+1/p^{\prime}}}.
\end{equation}
From the additional information, one sees that the inequality \eqref{e3.4} also holds for $p=\infty$.

On the other hand, 
\begin{equation}\label{e3.5}
\left|\int_{x_0}^{a-\delta}(q-\hat{q})(x)f_{0}(x)dx\right| 
\leq C_{2}e^{2|\textnormal{Im}z|(a-\delta)},
\end{equation}
for some constant $C_2>0$. Moreover, the function $g(z)$ in \eqref{e3.3} has the following estimate:
\begin{equation}\label{e3.6}
 |g(z)|
\leq C_{3}
\frac{e^{2|\textnormal{Im}z| a}}{|\textnormal{Im}z|^{k+1/p^{\prime}}} |z|^{\gamma}e^{-(a-x_0)|\textnormal{Im}z|}
\leq C_{3} 
\frac{e^{2|\textnormal{Im}z|a}}{|\textnormal{Im}z|^{k+1/p^{\prime}}} |z|^{\gamma}e^{-\delta|\textnormal{Im}z|},
\end{equation}
for some constant $C_3>0$. By using \eqref{e3.3}, \eqref{e3.4}, \eqref{e3.5} and \eqref{e3.6}, we complete the proof.
\end{proof}

\section{Proofs of Theorems \ref{t4.1}--\ref{t4.3}}\label{s4}

Before proving the main Theorems \ref{t4.1}--\ref{t4.3}, we first present a lemma concerning the derivatives of $F(\lambda)$.

\begin{lemma}\label{l4.4}
Given $f \in L^1(0,1)$, $a \in(0,1)$, define
\begin{align*}
F(\lambda)=\int_{0}^{a}\cos(2\sqrt{\lambda}x)f(x)dx.
\end{align*}
Then
\begin{align*}
F^{\prime}(\lambda)=-\int_{0}^{a} \sin(2\sqrt{\lambda} x) \frac{x}{\sqrt{\lambda}}f(x)dx,~~
\end{align*}
and 
\begin{align*}
F^{(k)}(0)=(-1)^k \frac{4^k(k)!}{(2k)!}
\int_{0}^{a}x^{2k}f(x)dx, 
~~k=0,1,2, \cdots.
\end{align*}
\end{lemma}

The proof of Lemma \ref{l4.4} is straightforward, so we omit it. We are now ready to prove Theorem \ref{t4.1}.

\begin{proof}[Proof of Theorem \ref{t4.1}] 
We first consider the case $0 \notin \Lambda$. 
Suppose that there exists another potential $\hat{q} \in L^p(0,1)$ such that $q(x)=\hat{q}(x)$, a.e. on $(a, 1)$, 
$\lambda_n \in \sigma(\ell, q, \beta_n) 
\cap \sigma(\ell, \hat{q}, \beta_n)$ 
for $\lambda_n \in \Lambda$, and 
$\zeta(\lambda_{n},\beta_{n}, q)$=$\zeta(\lambda_{n},\beta_{n}, \hat{q})$
for $\lambda_n \in S$.

If the values $\lambda_{n}$ have a finite accumulation point, then the entire function $\lambda^{\ell+1} H_{\ell}(\lambda) \equiv 0$. Moreover, the expression of $\lambda^{\ell+1} H_{\ell}(\lambda)$ in \eqref{l3.7} gives 
$$\int_{0}^{a}(q(x)-\hat{q}(x))dx=0,$$
which implies that 
\begin{align*}
\int_{0}^{a}\cos(2\sqrt{\lambda}x)(A_{\hat{q}}(q-\hat{q})(x))dx\equiv0,\;\lambda \in \mathbb{C}.
\end{align*}
Therefore, one has $A_{\hat{q}}(q-\hat{q})(x)\equiv0$, which, together with Remark \ref{r3.3}, implies that $q(x)=\hat{q}(x)$ a.e. on $(0, a)$.

If the values $\lambda_n$ have no finite accumulation point, by Lemma \ref{l3.1} and Remark \ref{r3.2}, we have
\begin{equation}\label{e4.12}
\lambda^{\ell+1}H_{\ell}(\lambda)
=\frac{1}{2}\int_{0}^{a} 
\cos(2\sqrt{\lambda}x)(A_{\hat{q}}(q-\hat{q})(x))dx,
\end{equation}
where $A_{\hat{q}}(q-\hat{q})\in L^{p}(0,a)$. 
Substituting $\lambda=\lambda_n$ into the identity \eqref{e4.12}, we obtain 
\begin{equation}\label{e4.13}
\int_{0}^{a}\cos(2\sqrt{\lambda_{n}}x)(A_{\hat{q}}(q-\hat{q})(x))dx=0,\; \lambda_n \in \Lambda.  
\end{equation}
From Lemmas \ref{l3.4} and \ref{l4.4}, together with the identity \eqref{e4.12}, one has 
\begin{equation}\label{e4.14}
\int_{0}^{a}x\sin(2\sqrt{\lambda_n}x)(A_{\hat{q}}(q-\hat{q})(x))dx=0,\; \lambda_n \in S, 
\end{equation}
and
\begin{align}\label{e4.1}
\int_{0}^{a}x^{2 k}(A_{\hat{q}}(q-\hat{q})(x))dx=0,~~k=0,1,2, \cdots, \ell.
\end{align}
Since the system $S_{\ell}(\Lambda, S)$ is closed in $L^p(0, a)$, the facts \eqref{e4.13}--\eqref{e4.1} give $A_{\hat{q}}(q-\hat{q})(x)=0$ a.e. on $(0,a)$. 
Then Remark \ref{r3.3} implies $q(x)=\hat{q}(x)$ a.e. on $(0, a)$.

For the case $0 \in \Lambda$ and $0 \notin S$ (or the case $0 \in S$), the only difference in the proof is that 0 is a zero of $\lambda^{\ell+1}H_{\ell}(\lambda)$ with  multiplicity $\ell+2$ (or $\ell+3$). Thus, the equality \eqref{e4.1} is changed to
\begin{gather*}
\int_{0}^{a} x^{2k}(A_{\hat{q}}(q-\hat{q})(x))dx=0,
~~k=0,1,2, \cdots, \ell+1,\\
(\text{or}~~\int_{0}^{a} x^{2k}(A_{\hat{q}}(q-\hat{q})(x))dx=0, 
~~k=0,1,2, \cdots, \ell+2,)
\end{gather*}
and the other parts of the proof remain the same. This completes the proof.
\end{proof}

In order to prove Theorem \ref{t4.2}, we need the following auxiliary lemma.
\begin{lemma}\cite[Lemma 2.1]{Horv2005}\label{l4.5}
Let $B_1$ and $B_2$ be Banach spaces.  Suppose for every $q \in B_1$, a continuous linear operator
\begin{align*}
A_q: B_1 \rightarrow B_2
\end{align*}
is defined so that for some $q_0 \in B_1$, 
\begin{align*}
A_{q_0}: B_1 \rightarrow B_2
\end{align*}
is an isomorphism, and the mapping $q \rightarrow A_q$ is Lipschitzian in the sense that
\begin{align}\label{e4.2}
\|(A_{q_1}-A_{q_2})h\| 
\leq c(q_0)\|q_1-q_2\|\|h\|,
\end{align}
for all $h$, $q_1$, $q_2\in B_1$ and $\|q_1\|$, $\|q_2\| \leq 2\|q_0\|$,
where the constant $c(q_0)$ is independent of $h, q_1, q_2$. Then the set $\{A_{q_1}(q_1-q_0)\mid q_1\in B_1\}$ contains a ball in $B_2$ with center at the origin.
\end{lemma}

For our purposes, let $B_1=B_2=L^p(0, a)$. The continuous linear operator $A_{\hat{q}}$ defined in \eqref{l3.8} satisfies the Lipschitzian condition \eqref{e4.2} (see \cite[Lemma 6.1]{XYB2023}). In order to apply the conclusion of Lemma \ref{l4.5}, we assume that the continuous linear operator $A_q$ is an isomorphism. Next, we proceed to prove  Theorem \ref{t4.2}. 

\begin{proof}[Proof of Theorem \ref{t4.2}]
We first consider the case $0 \notin \Lambda$. 
Suppose $q \in L^p(0,1)$. We will prove that there exists a different $\hat{q} \in L^p(0,1)$, with $\hat{q}=q$ a.e. on $(a, 1)$, such that for the given sets $\Lambda$ and $S$, there exist boundary parameters $\{\beta_n\}$ for which $\lambda_n \in \sigma(\ell, q, \beta_n) \cap \sigma(\ell, \hat{q}, \beta_n)$ for $\lambda_n \in \Lambda$, and $\zeta(\lambda_{n},\beta_{n}, q) = \zeta(\lambda_{n},\beta_{n}, \hat{q})$ for $\lambda_n \in S$.

Since the set $S_{\ell}(\Lambda, S)$ is not closed in $L^p(0, a)$, then there exists $h \in L^p(0, a)$, $h \neq 0$, such that
\begin{align}
\int_{0}^{a}h(x)\cos(2\sqrt{\lambda_n}x)dx=0, 
~~\lambda_n \in \Lambda, \label{e4.3}\\
\int_{0}^{a}h(x)x\sin(2\sqrt{\lambda_n}x)dx=0, 
~~\lambda_n \in S, \label{e4.4}\\
\int_{0}^{a}h(x)x^{2k}dx=0, 
~~k=0,1,2, \cdots, \ell. \label{e4.5}
\end{align}
By Lemma \ref{l4.5}, there exists a sufficiently small constant $\eta$ with $\eta>0$, such that $\eta h \in \{A_{q_1}(q_1-q)\mid q_1 \in L^p(0, a)\}$. Consequently, there exists $\hat{q} \in L^p(0, a)$, such that
\begin{align}\label{e4.6}
A_{\hat{q}}(\hat{q}-q)=\eta h.
\end{align}
Since $h\neq 0$, one has $q-\hat{q}\neq 0$.

The equalities \eqref{l3.7} and \eqref{e4.6} give 
\begin{align}\label{e4.8}
\lambda^{\ell+1} H_{\ell}(\lambda)
=\frac{1}{2}\int_{0}^{a}(q-\hat{q})(x)dx
+\frac{1}{2}\eta\int_{0}^{a}h(x)\cos(2\sqrt{\lambda}x)dx.
\end{align}
Putting $\lambda=0$ into \eqref{e4.8} and considering \eqref{e4.5}, we obtain 
\begin{align*}
\int_{0}^{a}(q-\hat{q})(x)dx=-\eta\int_{0}^{a}h(x)dx=0.
\end{align*}
 Then the equality \eqref{e4.8} becomes 
\begin{align}\label{e4.9}
\lambda^{\ell+1} H_{\ell}(\lambda)
=\frac{1}{2}\eta\int_{0}^{a}h(x)\cos(2\sqrt{\lambda}x)dx.
\end{align}

The equalities \eqref{e4.5} for $k=0,1,2, \cdots, \ell$, guarantee that the function $\lambda^{\ell+1}H_{\ell}(\lambda)$ has a zero $\lambda=0$ with multiplicity $\ell+1$. 
The equalities \eqref{e4.3} and \eqref{e4.4} imply that 
\begin{align*}
\lambda_n^{\ell+1} H_{\ell}(\lambda_n)
=\frac{1}{2}\eta\int_{0}^{a}h(x)\cos(2\sqrt{\lambda_n}x)dx=0, ~~\lambda_n \in \Lambda, \\
\lambda_n^{\ell+1}\dot{H}_{\ell}(\lambda_n)
=-\frac{1}{2}\eta\int_{0}^{a}h(x) \frac{x\sin(2\sqrt{\lambda_n}x)}{\sqrt{\lambda_n}}dx=0, 
~~\lambda_n \in S .
\end{align*}
Consequently, there exists a $\beta_n$ such that $\lambda_n \in \sigma(\ell, q, \beta_n) \cap \sigma(\ell, \hat{q}, \beta_n)$ if $\lambda_n \in \Lambda$, and
$\zeta(\lambda_{n},\beta_{n}, q)$=$\zeta(\lambda_{n},\beta_{n}, \hat{q})$
if $\lambda_n \in S$. 
So the proof for the case $0 \notin \Lambda$ is finished.

Now, we suppose $0 \in \Lambda$ and $0 \notin S$. The only difference from the case $0 \notin \Lambda$ is that we should additionally prove 
$0\in\sigma(\ell,q,\beta)\cap\sigma(\ell,\hat{q},\beta)$ for certain $\beta\in\mathbb{R}\cup\{\infty\}$.

Since $S_{\ell}(\Lambda, S)$ is not closed, we choose $h \neq 0$ in $L^p(0, a)$ such that \eqref{e4.3}--\eqref{e4.4} hold and
\begin{align}\label{e4.10}
\int_{0}^{a} x^{2k}h(x)dx=0, 
~~k=0,1,2, \cdots, \ell+1.
\end{align}
All other steps are the same as in the proof of the case $0 \notin \Lambda$. The equality \eqref{e4.10} gives that $0$ is a zero of $\lambda^{\ell+1}H_{\ell}(\lambda)$ with multiplicity $\ell+2$, which implies $H_{\ell}(0)=0$, that is,  
$0\in\sigma(\ell,q,\beta)\cap\sigma(\ell,\hat{q},\beta)$ for certain $\beta\in\mathbb{R}\cup\{\infty\}$.

Finally, suppose $0\in S$. The only difference from other cases is that we should additionally prove $0\in \sigma(\ell, q, \beta) \cap \sigma(\ell, \hat{q}, \beta)$ and $\zeta(0,\beta, q)$=$\zeta(0,\beta, \hat{q})$ for certain $\beta\in\mathbb{R}\cup\{\infty\}$.

Since $S_{\ell}(\Lambda, S)$ is not closed, we choose $h \neq 0$ in $L^p(0, a)$ such that \eqref{e4.3}--\eqref{e4.4} hold and
\begin{align}\label{e4.11}
\int_{0}^{a} x^{2k}h(x)dx=0, 
~~k=0,1,2, \cdots, \ell+2.
\end{align}
All other steps are the same as in the proof of the case $0 \notin \Lambda$. The equality \eqref{e4.11} gives that $0$ is a zero of $\lambda^{\ell+1}H_{\ell}(\lambda)$ with multiplicity $\ell+3$, which implies $H_{\ell}(0)=0$ and $\dot{H}_{\ell}(0)=0$, that is,  
$0\in\sigma(\ell,q,\beta)\cap\sigma(\ell,\hat{q},\beta)$
and 
$\zeta(0,\beta, q)$=$\zeta(0,\beta, \hat{q})$ for certain $\beta\in\mathbb{R}\cup\{\infty\}$. The proof is complete.
\end{proof}

\begin{remark}\label{r4.6}
The hypothesis that the operator $A_q$ is an isomorphism is very strong and it may not hold; see \cite[Remark 5.1]{XYB2023}.
\end{remark}

At the end of this section, we present the proof of Theorem \ref{t4.3}.

\begin{proof}[Proof of Theorem \ref{t4.3}]
Proof by contradiction. Suppose $\hat{q}\neq q$. Then 
$z^{2\ell+2}H_{\ell}(z^2)$ cannot be identically zero.
This is due to the fact that, if $z^{2\ell+2}H_{\ell}(z^2)\equiv0$, then
\begin{equation*}
\frac{\phi_{\ell}^{\prime}(z^2,1,q)}
{\phi_{\ell}(z^2,1,q)}
\equiv\frac{\phi_{\ell}^{\prime}(z^2,1,\hat{q})}
{\phi_{\ell}(z^2,1,\hat{q})},
\end{equation*}
which, together with the uniqueness theorem of the inverse spectral problem based on the Weyl-Titchmarsh $m$-function (see \cite{KST2010}), implies that $q=\hat{q}$ a.e. on $(0,1)$.

Without loss of generality, let $0\notin\Lambda$, that is, $H_{\ell}(0)\neq 0$. Then $H_{\ell}(z^2)$ is a nontrivial entire function of $z$ with zeros at $\pm \sqrt{\lambda_n}$ and at least double zeros for $\lambda_n \in S$.

Recall the Jensen formula: Let $f(z)$ be analytic for $|z|<R$ and $f(0) \neq 0$. If $n(t)$ is the number of zeros of $f(z)$ in $|z| \leq t$, then for $0<r<R$, 
\begin{align*}
\int_{0}^{r} \frac{n(t)}{t} d t
=\frac{1}{2\pi}\int_{0}^{2\pi}\ln|f(re^{i\varphi})|d\varphi-\ln|f(0)|.
\end{align*}

By applying the Jensen formula to $H_{\ell}(z^2)$, we get, with the notation $m(t)=2n_{\Lambda}(t^2)+2n_S(t^2)$, that
\begin{equation}\label{e4.15}
\int_{0}^{r}\frac{m(t)}{t}dt 
\leq \frac{1}{2\pi} \int_{0}^{2\pi} 
\ln|H_{\ell}(r^2 e^{2i\varphi})|d\varphi+C_4, 
\end{equation}
where $C_4=| \ln | H_{\ell}(0)| |$. Substituting the estimate of $H_{\ell}(z^2)$ from the inequality \eqref{e3.7} into \eqref{e4.15}, after some simple calculations, it can be seen that 
\begin{equation}\label{e4.16}
\begin{split}
& \int_{0}^{r}\frac{m(t)}{t}dt-\frac{4ra}{\pi}+\left(k+2\ell+2+\frac{1}{p^{\prime}}\right)\ln r  \\
\leq & \frac{1}{2\pi}\int_{0}^{2\pi}
\ln \left(\varepsilon+Cr^{\gamma}e^{-\delta r|\sin \varphi|}
+C(r|\sin \varphi|)^{k+\frac{1}{p^{\prime}}}
e^{-2\delta r|\sin \varphi|}\right)d\varphi
+C_4.
\end{split}
\end{equation}

It can be easily verified that
$$ r^{\gamma}e^{-\delta r|\sin \varphi|}\leq\frac{(\gamma e)^{\gamma}}{\delta^{\gamma}|\sin\varphi|^{\gamma}},$$
$$(r|\sin\varphi|)^{k+1/p^{\prime}}e^{-2\delta r|\sin\varphi|} \leq \left(\frac{k+1/p^{\prime}}{2\delta}\right)^{k+1/p^{\prime}} e^{-(k+1/p')}.$$
By Lebesgue's dominated convergence theorem, we have
\begin{equation}\label{e4.17}
\lim_{r\rightarrow\infty}\frac{1}{2\pi} 
\int_{0}^{2 \pi} 
\ln (\varepsilon+cr^{\gamma}e^{-\delta r|\sin \varphi|}
+c(r|\sin \varphi|)^{k+\frac{1}{p^{\prime}}}
e^{-2\delta r|\sin \varphi|})
d\varphi
=\ln \varepsilon.
\end{equation}

Taking the upper limit as $r\rightarrow\infty$ on both sides of the inequality \eqref{e4.16}, and considering \eqref{e4.17}, we have
\begin{equation}\label{e4.18}
\limsup_{r\rightarrow\infty}\int_{0}^{r}\frac{m(t)}{t}dt
-\frac{4ra}{\pi}+\left(k+2\ell+2+\frac{1}{p^{\prime}}\right)\ln r 
\leq\ln\varepsilon+C_4.
\end{equation}
From the arbitrariness of $\varepsilon$, the inequality \eqref{e4.18} implies 
\begin{align*}
\int_{0}^{r}\frac{m(t)}{t}dt
-\frac{4ra}{\pi}
+\left(k+2 \ell+2+\frac{1}{p^{\prime}}\right)\ln r 
\rightarrow-\infty, 
~~r\rightarrow\infty,
\end{align*}
which contradicts the assumptions. Therefore, $q=\hat{q}$ a.e. on $(0,a)$. This completes the proof.
\end{proof}

\section{Some corollaries of Theorem \ref{t4.3}}\label{s5}

In this section, we present some corollaries of Theorem \ref{t4.3}, which provide uniqueness results for the inverse spectral problem using different types of known data. For a fixed boundary condition, given a partial potential $q$ on a certain interval $(a,1)$, we give a uniqueness result based on a sequence of eigenvalues in Corollary \ref{t5.3}; a uniqueness result based on a pair of eigenvalues and the corresponding norming constants in Corollary \ref{t5.2}; and a uniqueness result for the half-inverse problem with a certain smoothness condition on the potential in Corollary \ref{t5.6}. For spectral data under two different boundary conditions with unpaired eigenvalues and norming constants, uniqueness results are given in Corollaries  \ref{t5.7}, \ref{t5.9} and \ref{t5.8}.

Denote
\begin{align*}
L_{\ell}^{1}(0,1)=
\begin{cases}
L^1(0,1), &\ell>-1/2, \\
\left\{f \in L^1(0,1)\big|\int_0^1|(1-\ln x) f(x)|dx<\infty\right\}, &\ell=-1/2.
\end{cases}
\end{align*}
Given $q\in L_{\ell}^{1}(0, 1)$, from \eqref{e1.3}, \eqref{e1.4} and \eqref{e1.7}, one sees that the eigenvalue $\lambda_{\ell, \beta, n}(q)\in\sigma(\ell, q, \beta)$ has the following asymptotic behavior:
\begin{equation}\label{e5.4}
\sqrt{\lambda_{\ell, \infty, n}(q)}=\left(n+\frac{\ell}{2}\right) \pi+O\left(\frac{1}{n}\right), ~~n \in \mathbb{N},
\end{equation}
\begin{equation}\label{e5.5}
\sqrt{\lambda_{\ell, \beta, n}(q)}=\left(n+\frac{\ell-1}{2}\right) \pi+O\left(\frac{1}{n}\right), ~~\beta \in \mathbb{R}, 
~~n \in \mathbb{N}.
\end{equation}

\begin{corollary}\label{t5.3}
Let $\ell\geq-1/2$, $a \in(0,1/2]$, and $q$, $\hat{q}\in L_{\ell}^{1}(0,1)$. Consider a subset $S\subset\sigma(\ell,q,\beta)\cap\sigma(\ell,\hat{q},\beta)$. Assume that $q=\hat{q}$ a.e. on $(a, 1)$. If the set $S$ satisfies
\begin{equation}\label{e5.8}
n_S(t) \geq 
\begin{cases}
2a n_{\sigma(\ell, q, \beta)}(t)+(a-1)\ell,& \beta=\infty,\\
2a n_{\sigma(\ell, q, \beta)}(t)+(a-1)\ell-2a,&\beta\in\mathbb{R},
\end{cases}
\end{equation}
for all sufficiently large $t>0$, then $q=\hat{q}$ a.e. on $(0,a)$.
\end{corollary}

\begin{corollary}\label{t5.2}
Let $\ell\geq-1/2$, $a \in(0,1]$, and $q$, $\hat{q}\in L_{\ell}^{1}(0,1)$. Consider a subset $S\subset\sigma(\ell,q,\beta)\cap\sigma(\ell,\hat{q},\beta)$ such that 
$\zeta(\lambda_{n},\beta, q)$=$\zeta(\lambda_{n},\beta, \hat{q})$
for $\lambda_n \in S$. Assume that $q=\hat{q}$ a.e. on $(a, 1)$. If the set $S$ satisfies
\begin{equation}\label{e5.3}
n_S(t) \geq 
\begin{cases}
a n_{\sigma(\ell, q, \beta)}(t)+\frac{(a-1) \ell}{2},& \beta=\infty,\\
a n_{\sigma(\ell, q, \beta)}(t)+\frac{(a-1) \ell}{2}-\frac{a}{2},
&\beta\in\mathbb{R},
\end{cases}
\end{equation}
for all sufficiently large $t>0$, then $q=\hat{q}$ a.e. on $(0,a)$.
\end{corollary}

For the case $\ell=0$, the result of Corollary \ref{t5.3} is due to Gesztesy and Simon \cite[Theorem 1.3]{GSb2000}, and the result of Corollary \ref{t5.2} is due to Wei and Xu \cite[Theorem 4.1]{WX2012}. 
Similar results in Corollaries \ref{t5.3} and \ref{t5.2} with $\ell\geq0$ under Dirichlet and Neumann boundary conditions can be found in \cite[Theorems 4.1 and 4.2]{YGJ2019}.

We only prove Corollary \ref{t5.2}, the proof of Corollary \ref{t5.3} is analogous to that of Corollary \ref{t5.2}, therefore, it is omitted.  Before proving Corollary  \ref{t5.2}, we first present the following lemma.

\begin{lemma}\cite[Lemma 3.2]{HorvO2016}\label{l5.1}
Let $\alpha_{1}>0$, $\alpha_{2}\in\mathbb{R}$, $\mu_0 \leq \mu_1 \leq \mu_2 \leq \cdots$ be real numbers tending to $+\infty$, and $m(t)=\sum\limits_{\mu_k \leq t}1$. If $\sqrt{\mu_k} \leq \alpha_{1}k+\alpha_{2}+O(1/k)$ holds for all sufficiently large indices $k$ then
\begin{align*}
\int_1^R \frac{2 m\left(t^2\right)}{t} d t \geq \frac{2}{\alpha_{1}} R+\left(1-2 \frac{\alpha_{2}}{\alpha_{1}}\right) \ln R+O(1), R \rightarrow \infty.
\end{align*}
\end{lemma}

\begin{proof}[Proof of Corollary \ref{t5.2}]
We first consider the case $\beta=\infty$. Let $m(t)=4n_S(t^2)$. The condition \eqref{e5.3} for $\beta=\infty$ gives  
\begin{equation}\label{e5.6}
\begin{split}
& \int_{0}^{r}\frac{m(t)}{t}dt
-\frac{4ra}{\pi}
+(2\ell+2)\ln r \\
\geq & 2a\int_{1}^{r} 
\frac{2n_{\sigma(\ell,q,\infty)}(t^2)}{t}dt
+2a\ell\ln r
-\frac{4ra}{\pi}+2\ln r.
\end{split}
\end{equation}
The asymptotic formula \eqref{e5.4} and Lemma \ref{l5.1} yield 
\begin{align}\label{e5.1}
\int_{1}^{r} \frac{2 n_{\sigma(\ell,q,\infty)}(t^2)}{t}dt 
\geq \frac{2}{\pi}r-(\ell+1)\ln r+O(1), 
~~r \rightarrow \infty.
\end{align}
From the inequalities \eqref{e5.6} and \eqref{e5.1}, one has 
\begin{equation*}
\int_{0}^{r}\frac{m(t)}{t}dt
-\frac{4ra}{\pi}
+(2\ell+2)\ln r >-\infty, 
~~r\rightarrow\infty.
\end{equation*}
By Theorem \ref{t4.3}, we have $q=\hat{q}$ a.e. on $(0,a)$. 

Now we consider the case $\beta\in\mathbb{R}$. The fact that $m(t)=4n_S(t^2)$ and the condition \eqref{e5.3} for $\beta\in\mathbb{R}$ imply   
\begin{equation}\label{e5.7}
\begin{split}
& \int_{0}^{r}\frac{m(t)}{t}dt
-\frac{4ra}{\pi}
+(2\ell+2)\ln r \\
\geq & 2a\int_{1}^{r} 
\frac{2n_{\sigma(\ell,q,\beta)}(t^2)}{t}dt
+2a(\ell-1)\ln r
-\frac{4ra}{\pi}+2\ln r.
\end{split}
\end{equation}
The asymptotic behavior \eqref{e5.5} and Lemma \ref{l5.1} yield 
\begin{align}\label{e5.2}
\int_{1}^{r} \frac{2 n_{\sigma(\ell,q,\beta)}(t^2)}{t}dt 
\geq \frac{2}{\pi}r-\ell\ln r+O(1), 
~~r \rightarrow \infty.
\end{align}
From the inequalities \eqref{e5.7} and \eqref{e5.2}, we have
\begin{equation*}
\int_{0}^{r}\frac{m(t)}{t}dt
-\frac{4ra}{\pi}
+(2\ell+2)\ln r  >-\infty, 
~~r\rightarrow\infty.
\end{equation*}
By Theorem \ref{t4.3}, we have $q=\hat{q}$ a.e. on $(0,a)$. This completes the proof.
\end{proof}

\begin{remark}\label{r5.1}
For $\ell=-1/2$ and $q$, $\hat{q}\in L^1(0,1)$,
since $\alpha_{1}=\pi$ and 
\begin{align*}
\alpha_{2}=
\begin{cases}
(\frac{\ell}{2}+1)\pi+c, &\beta=\infty, \\
\frac{\ell+1}{2}\pi+c, &\beta\in\mathbb{R},
\end{cases}
\end{align*}
for a sufficiently small $c>0$ with the symbols $\alpha_{1}$ and $\alpha_{2}$ in Lemma \ref{l5.1}, the result of Corollary \ref{t5.3} is valid except at $a = 1/2$; however,  it holds for $a = 1/2$ if any positive number is added to the right-hand side of the inequality \eqref{e5.8}.
The result of Corollary \ref{t5.2} also holds except at $a = 1$; however, if any positive number is added to the right-hand side of the inequality \eqref{e5.3}, then it also holds for $a = 1$. 

\end{remark}

For $x\in\mathbb{R}$, we use the symbol $[x]$ to denote the greatest integer not exceeding $x$, and present the following corollary.

\begin{corollary}\label{t5.6}
Let $\ell\geq-1/2$ and $q\in L_{\ell}^{1}(0,1)$. 
Suppose that $q$ is $C^{2 k}((1/2-\varepsilon, 1/2+\varepsilon))$ for some $k\in\mathbb{N}\cup\{0\}$ and $\varepsilon>0$. Then the following conclusions hold.
\begin{enumerate}[label=\textnormal{(\arabic*).}]
\item The potential $q$ on $(1/2,1)$, and the set $\sigma(\ell,q,\infty)$ except for the number of $\left([\ell/2]+k+1\right)$ elements, uniquely determine $q$. 
\item The potential $q$ on $(1/2,1)$, and the set $\sigma(\ell,q,\beta)$ except for the number of $\left([(\ell+1)/2]+k+1\right)$ elements, uniquely determine $q$, where $\beta\in\mathbb{R}$. 
\end{enumerate}
\end{corollary}

For the case $\ell=0$, the result of Corollary \ref{t5.6} is due to Gesztesy and Simon \cite[Theorem 1.2]{GSb2000}. 

\begin{proof}[Proof of Corollary \ref{t5.6}]
(1).
Suppose $S\subseteq\sigma(\ell,q,\infty)$ such that the number of elements in the set $\sigma(\ell,q,\infty)\setminus S$ is less than $[\ell/2]+k+1$, 
and let $m(t)=2n_S(t^2)$. Then one has 
\begin{equation}\label{e5.10}
m(t)\geq 2n_{\sigma(\ell, q,\infty)}(t^2)-\ell-2k-2,
\end{equation}
for all sufficiently large $t>0$. The inequalities \eqref{e5.1} and \eqref{e5.10} imply 
\begin{align*}
\int_{0}^{r}\frac{m(t)}{t}dt-\frac{2r}{\pi}+(2k+2\ell+3)\ln r
\geq&\int_{0}^{r}\frac{2n_{\sigma(\ell, q,\infty)}(t^2)}{t}dt
-\frac{2r}{\pi}+(\ell+1)\ln r\\
>&-\infty, ~~r\rightarrow\infty.
\end{align*}
Therefore, the conclusion (1) is proved by Theorem \ref{t4.3}.

(2). Let $\beta\in\mathbb{R}$. 
Suppose $S\subseteq\sigma(\ell,q,\beta)$ such that the number of elements in the set $\sigma(\ell,q,\beta)\setminus S$ is less than $[(\ell+1)/2]+k+1$, 
and let 
$m(t)=2n_S(t^2)$. Then one has  
\begin{equation}\label{e5.11}
m(t)\geq 2n_{\sigma(\ell, q,\beta)}(t^2)-\ell-2k-3,
\end{equation}
for all sufficiently large $t>0$. The inequalities \eqref{e5.2} and \eqref{e5.11} imply 
\begin{align*}
\int_{0}^{r}\frac{m(t)}{t}dt
-\frac{2r}{\pi}+(2k+2\ell+3)\ln r
\geq&\int_{0}^{r}\frac{2n_{\sigma(\ell, q,\beta)}(t^2)}{t}dt
-\frac{2r}{\pi}+\ell\ln r\\
>&-\infty, ~~r\rightarrow\infty.
\end{align*}
Thus, the conclusion (2) is proved by Theorem \ref{t4.3}.
\end{proof}

\begin{remark}\label{r5.3}
For $\ell=-1/2$, if $q \in L^{1}(0,1)$ instead of $q \in L^{1}_{\ell}(0,1)$, then the result of Corollary \ref{t5.6} holds if an additional eigenvalue is added to the original condition.
\end{remark}

\begin{corollary}\label{t5.7}
Let $\ell\geq-1/2$, $q\in L_{\ell}^{1}(0,1)$, $\beta_{1}$, $\beta_{2}\in\mathbb{R}\cup\{\infty\}$, $\beta_{1}\neq\beta_{2}$, and $M\subseteq\mathbb{N}$. Then the following conclusions hold.
\begin{enumerate}[label=\textnormal{(\arabic*)}.]
\item If $\beta_{1}=\infty$ or $\beta_{2}=\infty$, then the sets $\sigma(\ell, q, \beta_1)$, $\{\zeta(\lambda_{\ell, \beta_{1}, n}(q),\beta_{1}, q)\}_{n\in M}$ and $\{\lambda_{\ell, \beta_{2}, n}(q)\}_{n\in\mathbb{N}\setminus M}$ uniquely determine $q$.
\item If $\beta_{1}\neq\infty$ and $\beta_{2}\neq\infty$, then the sets $\sigma(\ell, q, \beta_1)$, $\{\zeta(\lambda_{\ell, \beta_{1}, n}(q),\beta_{1}, q)\}_{n\in M}$ and $\{\lambda_{\ell, \beta_{2}, n}(q)\}_{n\in\mathbb{N}\setminus (M\cup\{k\})}$ for a positive integer $k\notin M$ uniquely determine $q$.
\end{enumerate}
\end{corollary}

For the case $\ell=0$, a similar result in Corollary \ref{t5.7} was given by Hatino\v{g}lu \cite[Theorems 4.2 and 4.11]{Hat2021}. Compared with Theorem 4.11 in \cite{Hat2021}, the required eigenvalue information in (2) of Corollary \ref{t5.7} can be reduced by one.

\begin{proof}[Proof of Corollary \ref{t5.7}]
It is well known that the elements in the two sets $\sigma(\ell, q, \beta_1)$ and $\sigma(\ell, q, \beta_2)$ appear interleaved on the real axis.
Therefore, one of the following two cases holds.

Case 1:
$$n_{\{\lambda_{\ell, \beta_{1}, n}(q)\}_{n\in S_0}}(t^{2})
    \geq n_{\{\lambda_{\ell, \beta_{2}, n}(q)\}_{n\in S_0}}(t^{2})
    \geq n_{\{\lambda_{\ell, \infty, n}(q)\}_{n\in S_0}}(t^{2}),$$
for all sufficiently large $t>0$, where the set $S_0$ is taken as $\mathbb{N}$, $M$, and $\mathbb{N}\setminus M$, respectively.

Case 2:
$$n_{\{\lambda_{\ell, \beta_{2}, n}(q)\}_{n\in S_0}}(t^{2})
    \geq n_{\{\lambda_{\ell, \beta_{1}, n}(q)\}_{n\in S_0}}(t^{2})
    \geq n_{\{\lambda_{\ell, \infty, n}(q)\}_{n\in S_0}}(t^{2}),$$
for all sufficiently large $t>0$, where the set $S_0$ is taken as $\mathbb{N}$, $M$, and $\mathbb{N}\setminus M$, respectively.

Let 
$$m(t)=2n_{\sigma(\ell, q, \beta_1)}(t^2)+2n_{\{\lambda_{\ell, \beta_{1}, n}(q)\}_{n\in M}}(t^2)+2n_{\{\lambda_{\ell,\beta_{2}, n}(q)\}_{n\in\mathbb{N}\setminus \tilde{M}}}(t^2).$$
We have 
\begin{align*}
m(t)\geq2n_{\sigma(\ell, q, \beta_1)}(t^2)+2n_{\{\lambda_{\ell, \beta_{1}, n}(q)\}_{n\in M}}(t^2)+2n_{\{\lambda_{\ell,\beta_{2}, n}(q)\}_{n\in\mathbb{N}\setminus M}}(t^2)-2\chi(\beta_{1},\beta_{2}),
\end{align*}
for all sufficiently large $t>0$,
where
\begin{align*}
\tilde{M}:=
\begin{cases}
M,
&~\textnormal{if}~\beta_{1}=\infty~\textnormal{or}~\beta_{2}=\infty,\\
M\cup\{k\},
&~\textnormal{if}~\beta_{1}\neq\infty~\textnormal{and}~\beta_{2}\neq\infty,
\end{cases}
\end{align*}
and
\begin{align}\label{e5.17}
\chi(\beta_{1},\beta_{2}):=
\begin{cases}
0,
~\textnormal{if}~\beta_{1}=\infty~\textnormal{or}~\beta_{2}=\infty,\\
1,
~\textnormal{if}~\beta_{1}\neq\infty~\textnormal{and}~\beta_{2}\neq\infty.
\end{cases}
\end{align}
Then, one has 
\begin{align}\label{e5.9}
m(t)\geq 
\begin{cases}
4n_{\sigma(\ell, q, \beta_2)}(t^2)-2\chi(\beta_{1},\beta_{2}),
~\textnormal{if~case~1~holds},\\
4n_{\sigma(\ell, q, \beta_1)}(t^2)-2\chi(\beta_{1},\beta_{2}),
~\textnormal{if~case~2~holds},
\end{cases}
\end{align}
for all sufficiently large $t>0$. The inequalities \eqref{e5.1}, \eqref{e5.2}, and \eqref{e5.9} give  
\begin{align*}
\int_{0}^{r}\frac{m(t)}{t}dt
-\frac{4r}{\pi}
+(2\ell+2)\ln r 
 >-\infty, 
~~r\rightarrow\infty.
\end{align*}
By Theorem \ref{t4.3}, we have  that the conclusions (1) and (2) hold.
This completes the proof.
\end{proof}

\begin{remark}\label{r5.4}
Consider the situation $\ell=-1/2$ and $q \in L^{1}(0,1)$. If $\beta_{1}\neq\infty$ and $\beta_{2}\neq\infty$, 
then the result of Corollary \ref{t5.7} holds provided that an additional eigenvalue is added to the original condition.
If $\beta_{1}=\infty$ or $\beta_{2}=\infty$, 
then the result of Corollary \ref{t5.7} holds provided that there exists a positive constant $c$ such that 
\begin{align*}
\sum_{\substack{k \in \mathbb{N}\setminus A \\ |\lambda_{\ell, \infty, k}|\leq r^{2}}} 
\int_{\sqrt{\lambda_{\ell,\beta,k}}}^{\sqrt{\lambda_{\ell, \infty, k}}}
\frac{1}{t}dt>c\ln r, 
~~r\rightarrow\infty.
\end{align*}
\end{remark}

In Corollary \ref{t5.7}, for the subset of the second spectrum $\sigma(l,q,\beta_2)$ that we have provided, we actually know the index of each eigenvalue, that is, we know which position a certain eigenvalue occupies in the sequence $\sigma(l,q,\beta_2)$. Next, in Corollaries \ref{t5.9} and \ref{t5.8}, we consider the situation where the eigenvalues in a subset of $\sigma(l,q,\beta_2)$ are given, but the indices of these eigenvalues are unknown; such conditions for the uniqueness of the inverse problem were  provided by Hatino\v{g}lu \cite{Hat2021} for the Sturm-Liouville operator.

\begin{corollary}\label{t5.9}
Let $\ell\geq-1/2$, $q\in L_{\ell}^{1}(0,1)$, $\beta_{1}$, $\beta_{2}\in\mathbb{R}\cup\{\infty\}$, $\beta_{1}\neq\beta_{2}$, 
$A=\{a_{n}\}_{n\in\mathbb{N}}\subseteq\sigma(\ell, q, \beta_1)$, and
$B=\{b_{n}\}_{n \in \mathbb{N}}\subseteq\sigma(\ell, q, \beta_2)$, 
such that 
\begin{align*}
\lim_{m \rightarrow \infty}\sum_{n=1}^m\frac{|A_{n, m}-A_{n}|}{ a_{n}^2}<\infty,
\end{align*}
and
$
\{A_{n}/a_{n}^2\}_{n \in \mathbb{N}}\in l^{1}
$,
where
\begin{align*}
A_{n, m}:=\frac{a_n}{b_n}\left(a_n-b_n\right) 
\prod_{\substack{j=1\\j \neq n}}^m \frac{a_j}{b_j} \frac{a_n-b_j}{a_n-a_j},
~~A_n:=\frac{a_n}{b_n}\left(a_n-b_n\right) 
\prod_{\substack{j=1 \\j \neq n}}^{\infty} \frac{a_j}{b_j} \frac{a_n-b_j}{a_n-a_j}.
\end{align*}
Then the sets $\sigma(\ell, q, \beta_1)$, 
$\{\zeta(a_{n},\beta_{1}, q)\}_{n\in \mathbb{N}}$ and 
$\sigma(\ell, q, \beta_2)\setminus B$ uniquely determine $q$.
\end{corollary}

For the case $\ell=0$, a similar result in Corollary \ref{t5.9} was given by Hatino\v{g}lu \cite[Theorems 4.6 and 4.13]{Hat2021}. There, to obtain uniqueness, an additional eigenvalue datum is required; moreover, when $\beta_{1}\neq\infty$ and $\beta_{2}\neq\infty$, there is an additional restriction on the selection of the eigenvalues. However, these conditions are not necessary in Corollary \ref{t5.9}. Before proving Corollary \ref{t5.9}, we first present the following two lemmas.

\begin{lemma}\label{l5.3}
Suppose that the meromorphic function $G$ is given by
\begin{align*}
G(z)=a z^2+b z+c+
\sum_{n\in\mathbb{N}} A_n\left(\frac{1}{z-a_n}+\frac{1}{a_n}\right),
\end{align*}
where $a$, $b$, $c$, $A_{n}\in\mathbb{R}$, and $\sum\limits_{n \in \mathbb{N}}|A_n|/a_n^2<\infty$. Then there exists a constant $C$ such that
\begin{align}\label{e5.16}
\int_0^r \frac{n_B\left(t^2\right)}{t} d t-\int_0^r \frac{n_A\left(t^2\right)}{t}dt \leq \frac{1}{2}\ln r+C,
\end{align}
for all sufficiently large $r>0$, where $B$ and $A$ are the sets of zeros and poles of $G$, respectively.
\end{lemma}

\begin{proof}
Recall the Jensen formula for the meromorphic function:
Let $f(z)$ be meromorphic for $|z|<R$, $f(0) \neq 0$ and $f(0) \neq \infty$. Then, for $0<r<R$, one has 
\begin{align*}
\int_0^r \frac{n_{\{z\mid f(z)=0\}}(t)}{t}dt
-\int_0^r \frac{n_{\{z\mid f(z)= \infty\}}(t)}{t}dt=\frac{1}{2\pi} \int_0^{2 \pi} \ln \left|f\left(r e^{i \varphi}\right)\right| d \varphi-\ln|f(0)|.
\end{align*}

Without loss of generality, let $G(0) \neq 0$ and $G(0) \neq \infty$. By applying the Jensen formula to $G$, we get
\begin{align*}
&\int_0^r \frac{n_B\left(t^2\right)}{t} d t-\int_0^r \frac{n_A\left(t^2\right)}{t}dt
=\frac{1}{2}\left(\int_0^{\sqrt{r}}\frac{n_B(t)}{t}dt
-\int_0^{\sqrt{r}} \frac{n_A(t)}{t}dt\right)\\
\leq& \frac{1}{4\pi} 
\int_0^{2\pi} \ln| G(\sqrt{r} e^{i\varphi})|d\varphi
+\frac{1}{2}|\ln| G(0)|| \\
\leq & \frac{1}{4\pi} 
\int_0^{2 \pi} \ln \left(|a|r+|b|\sqrt{r}+|c|
+\sum_{n\in\mathbb{N}} \frac{|A_n|\sqrt{r}}{|\sqrt{r} e^{i \varphi}-a_n||a_n|}\right)d\varphi
+\frac{1}{2}|\ln|G(0)|| \\
\leq& \frac{1}{4\pi} 
\int_0^{2 \pi}\ln\left(|a|r+|b|\sqrt{r}+|c|
+\frac{\sqrt{r}}{|\sin \varphi|}
\sum_{n\in\mathbb{N}}\frac{|A_n|}{a_n^2}\right)d\varphi
+\frac{1}{2}|\ln|G(0)||. 
\end{align*}
Since $\sum_{n \in \mathbb{N}}|A_n|/a_n^2<\infty$, there exist constants $C_{5}$, $C_{6}$, $C_{7}> 0$ that do not depend on $r$, such that
\begin{align*}
\int_0^r \frac{n_B(t^2)}{t} d t-\int_0^r \frac{n_A(t^2)}{t}dt
\leq \frac{1}{4\pi}
\int_0^{2 \pi}\ln\left(C_5+\frac{C_6}{|\sin \varphi|}\right)d\varphi+\frac{1}{2}\ln r+C_7,
\end{align*}
for all sufficiently large $r>0$. It can be easily verified that 
$\ln(C_5+C_6/|\sin \varphi|)$ is integrable on $(0,2\pi)$;  
therefore, we complete the proof. 
\end{proof}

\begin{lemma}\label{l5.4}
Let $\ell\geq-1/2$, $q\in L^{1}(0,1)$, $\beta_{1}$, $\beta_{2}\in\mathbb{R}\cup\{\infty\}$, $\beta_{1}\neq\beta_{2}$, 
$A=\{a_{n}\}_{n\in\mathbb{N}}\subseteq\sigma(\ell, q, \beta_1)$, and
$B=\{b_{n}\}_{n \in \mathbb{N}}\subseteq\sigma(\ell, q, \beta_2)$.
Then there exists a constant $C>0$ such that
\begin{align}\label{e5.19}
\sup_{n\in\mathbb{N}}(|b_{n}||b_{n}-a_{k}|)^{-1} \leq C|a_{k}|^{-1}
\end{align}
for all $k\in\mathbb{N}$.
\end{lemma}

\begin{proof}
If $\beta_{1}=\infty$ or $\beta_{2}=\infty$, from \eqref{e1.3} and \eqref{e1.4}, we can obtain
\begin{align*}
d=\inf_{n,k\in\mathbb{N}}|b_{n}-a_{k}|>0.
\end{align*}
If $\beta_{1}\neq\infty$ and $\beta_{2}\neq\infty$, from \eqref{e1.4}, \eqref{e1.5}, and \eqref{e1.6}, one can see that
\begin{align*}
\beta_{2}-\beta_{1}
=&\frac{\phi_{\ell}^{\prime}(\lambda_{\ell,\beta_1,n},1,q)}
{\phi_{\ell}(\lambda_{\ell,\beta_1,n},1,q)}
-\frac{\phi_{\ell}^{\prime}(\lambda_{\ell,\beta_2,n},1,q)}
{\phi_{\ell}(\lambda_{\ell,\beta_2,n},1,q)}\\
=&\frac{\phi_{\ell}^{\prime}\left(\lambda_{\ell, \beta_1, n}, 1, q\right) \phi_{\ell}\left(\lambda_{\ell, \beta_2, n}, 1, q\right)-\phi_{\ell}^{\prime}\left(\lambda_{\ell, \beta_2, n}, 1, q\right) \phi_{\ell}\left(\lambda_{\ell, \beta_1, n}, 1, q\right)}
{\phi_{\ell}(\lambda_{\ell,\beta_1,n},1,q) \phi_{\ell}(\lambda_{\ell, \beta_2, n}, 1, q)}\\
=& n(\sqrt{\lambda_{\ell, \beta_2, n}}
-\sqrt{\lambda_{\ell, \beta_1, n}})+o(1),
\end{align*}
and hence
$\lim_{n\rightarrow\infty}(\lambda_{\ell, \beta_2, n}-\lambda_{\ell, \beta_1, n})= 2(\beta_2-\beta_1).$
Then, one has
\begin{align*}
d=\inf_{n,k\in\mathbb{N}}|b_{n}-a_{k}|>0.
\end{align*}
Note that
$(|b_{n}||b_{n}-a_{k}|)^{-1} \leq (|b_{n}|||b_{n}|-|a_{k}||)^{-1}$.

If $|b_{n}|>|a_{k}|$, then
$(|b_{n}||b_{n}-a_{k}|)^{-1} \leq (d|a_{k}|)^{-1}$.

If $|b_{n}|<|a_{k}|$, then
\begin{align*}
(|b_{n}||b_{n}-a_{k}|)^{-1}
\leq&(|b_{n}|(|a_{k}|-|b_{n}|))^{-1} \\
\leq& \min\{((\min_{n\in\mathbb{N}}|b_{n}|)(|a_{k}|-(\min_{n\in\mathbb{N}}|b_{n}|)))^{-1}, (d(|a_{k}|-d))^{-1}\},
\end{align*}
and there exists a constant $C>d^{-1}$ such that
\begin{align*}
\limsup_{k\rightarrow\infty}|a_{k}|\min\{((\min_{n\in\mathbb{N}}|b_{n}|)(|a_{k}|-(\min_{n\in\mathbb{N}}|b_{n}|)))^{-1}, (d(|a_{k}|-d))^{-1}\}\leq C.
\end{align*}
Consequently, \eqref{e5.19} holds for all $k\in\mathbb{N}$. This completes the proof. 
\end{proof}

\begin{proof}[Proof of Corollary \ref{t5.9}]
Let 
\begin{align}\label{e5.15}
\begin{split}
m(t)
&=2n_{\sigma(\ell, q, \beta_1)}(t^2)
+2n_A(t^2)
+2n_{\sigma(\ell, q, \beta_2)\setminus B}(t^2)\\
&= 2n_{\sigma(\ell, q, \beta_1)}(t^2)
+2n_{\sigma(\ell, q, \beta_2)}(t^2)
-2(n_B(t^2)-n_A(t^2)).
\end{split}
\end{align}
From \eqref{e5.19}, the proof of \cite[Lemma 4.5]{Hat2021}, and Lemma \ref{l5.3}, we can conclude that the inequality \eqref{e5.16} holds for $A$ and $B$, where $A$ and $B$ are defined in Corollary \ref{t5.9}. Thus, the inequalities \eqref{e5.1}, \eqref{e5.2}, \eqref{e5.16} and \eqref{e5.15} give  
\begin{align*}
\int_{0}^{r}\frac{m(t)}{t}dt
-\frac{4r}{\pi}
+(2\ell+2)\ln r 
 >-\infty, 
~~r\rightarrow\infty.
\end{align*}
The prove is completed by Theorem \ref{t4.3}. 
\end{proof}

\begin{remark}\label{r5.6}
Let $\ell=-1/2$ and $q \in L^{1}(0,1)$. If $\beta_{1}\neq\infty$ and $\beta_{2}\neq\infty$, 
then the result of Corollary \ref{t5.9} also holds. 
If $\beta_{1}=\infty$ or $\beta_{2}=\infty$, then the result of Corollary \ref{t5.9} holds provided that an additional eigenvalue is added to the original condition.
\end{remark}

\begin{corollary}\label{t5.8}
Let $\ell\geq-1/2$, $q\in L_{\ell}^{1}(0,1)$, $\beta_{1}$, $\beta_{2}\in\mathbb{R}\cup\{\infty\}$, $\beta_{1}\neq\beta_{2}$, and
$A=\{a_{n}\}_{n\in\mathbb{N}}\subseteq\sigma(\ell, q, \beta_1)$, 
$B=\{b_{n}\}_{n \in \mathbb{N}}\subseteq\sigma(\ell, q, \beta_2)$, 
such that $\prod\limits_{n \in \mathbb{N}}(a_{n} / b_{n})$ is absolutely convergent. 
Then the following conclusions hold.
\begin{enumerate}[label=\textnormal{(\arabic*)}.]
\item If $\beta_{1}=\infty$ or $\beta_{2}=\infty$, then the sets $\sigma(\ell, q, \beta_1)$, 
    $\{\zeta(a_{n},\beta_{1}, q)\}_{n\in \mathbb{N}}$ and 
    $\sigma(\ell, q, \beta_2)\setminus B$ uniquely determine $q$.
\item If $\beta_{1}\neq\infty$ and $\beta_{2}\neq\infty$, then the sets  $\sigma(\ell, q, \beta_1)$, 
    $\{\zeta(a_{n},\beta_{1}, q)\}_{n\in \mathbb{N}}$ and 
    $\sigma(\ell, q, \beta_2)\setminus (B\cup\{\lambda_{\ell, \beta_{2}, k}(q)\})$ for $\lambda_{\ell, \beta_{2}, k}(q)\notin B$ uniquely determine $q$.
\end{enumerate}
\end{corollary}

For the case $\ell=0$, the result in Corollary \ref{t5.8} is due to Hatino\v{g}lu \cite[Theorems 4.8 and 4.14]{Hat2021}. The condition involving an additional restriction on the selection of eigenvalues in \cite[Theorem 4.14]{Hat2021} is not necessary in Corollary \ref{t5.8}. Moreover, compared with \cite[Theorem 4.14]{Hat2021}, the required eigenvalue information in (2) of Corollary \ref{t5.8} can be reduced by one. Before proving Corollary \ref{t5.8}, we first present the following lemma.

\begin{lemma}\label{l5.2}
Suppose the sequences $A=\{a_n\}_{n\in\mathbb{N}}$ and $B=\{b_n\}_{n\in\mathbb{N}}$ are such that $a_{n}=b_{n}+t_{n}b_{n}$, where $\sum_{n=1}^{\infty}|t_n|<\infty$. Then  there exists a  constant $C>0$, such that
\begin{equation}\label{e5.12}
\left|\int_0^r \frac{n_A(t^2)}{t}dt
-\int_0^r \frac{n_B(t^2)}{t}dt\right| \leq C,
\end{equation}
for all $r>0$.
\end{lemma}

\begin{proof}
Note that 
\begin{equation*}
\int_0^r \frac{n_A(t^2)}{t}dt
=\sum_{n=1}^{\infty}\max \left\{0,\ln r-\frac{1}{2}\ln|a_n|\right\}, 
\end{equation*}
and 
\begin{equation*}
\int_0^r \frac{n_B(t^2)}{t}dt
=\sum_{n=1}^{\infty}\max\left\{0,\ln r-\frac{1}{2}\ln|b_n|\right\}.
\end{equation*}
Let
\begin{align*}
d_n:=\max\left\{0,\ln r-\frac{1}{2}\ln|a_n|\right\}
-\max\left\{0,\ln r-\frac{1}{2} \ln|b_n|\right\},
\end{align*}
then, one has 
\begin{align*}
\left|\int_0^r \frac{n_A(t^2)}{t}dt
-\int_0^r \frac{n_B(t^2)}{t}dt\right|
=\left|\sum_{n=1}^{\infty}d_n\right| .
\end{align*}
Without loss of generality, let $1+t_n\neq0$ for any $n\in\mathbb{N}$. Then, for each $n$, we have
\begin{enumerate}
  \item [\textbullet]If $|b_n|<r^2$ and $|a_n|<r^2$, then $|d_n|=\frac{1}{2}\left|\ln\frac{|b_n|}{|a_n|}\right|=\frac{1}{2}|\ln|1+t_n||$.
  \item [\textbullet]If $|b_n|<r^2\leq|a_n|$, then $\frac{1}{2}\ln\frac{|b_n|}{|a_n|}\leq d_n<0$, so $|d_n|\leq \frac{1}{2}|\ln| 1+t_n||$.
  \item [\textbullet]If $|a_n|<r^2\leq|b_n|$, then $0<d_n\leq \frac{1}{2} \ln\frac{|b_n|}{|a_n|}$, so $|d_n|\leq\frac{1}{2}|\ln|1+t_n||$.
  \item [\textbullet]If $b_n \geq r^2$ and $a_n \geq r^2$, then $d_n=0$. 
\end{enumerate}
Consequently,
$|d_n|\leq\frac{1}{2}|\ln|1+t_n||$ for each $n$.
The condition $\sum\limits_{n=1}^{\infty}\left|t_n\right|<\infty$ gives 
\begin{align*}
\left|\sum_{n=1}^{\infty} d_n\right| 
\leq \sum_{n=1}^{\infty}|d_n| 
\leq \sum_{n=1}^{\infty}\frac{1}{2}|\ln|1+t_n||<\infty,
\end{align*}
that is, there exists a  constant $C>0$, such that \eqref{e5.12} holds. This completes the proof.
\end{proof}

\begin{proof}[Proof of Corollary \ref{t5.8}]
Let 
\begin{align*}
m(t)
=2n_{\sigma(\ell, q, \beta_1)}(t^2)
+2n_A(t^2)
+2n_{\sigma(\ell, q, \beta_2)\setminus \tilde{B}}(t^2).
\end{align*}
We have
\begin{align}\label{e5.13}
m(t)
=2n_{\sigma(\ell, q, \beta_1)}(t^2)
+2n_{\sigma(\ell, q, \beta_2)}(t^2)
+2n_A(t^2)-2n_B(t^2)-2\chi(\beta_{1},\beta_{2}),
\end{align}
for all sufficiently large $t>0$,
where
\begin{align*}
\tilde{B}:=
\begin{cases}
B,
&~\textnormal{if}~\beta_{1}=\infty~\textnormal{or}~\beta_{2}=\infty,\\
B\cup\{\lambda_{\ell, \beta_{2}, k}(q)\},
&~\textnormal{if}~\beta_{1}\neq\infty~\textnormal{and}~\beta_{2}\neq\infty,
\end{cases}
\end{align*}
and $\chi(\beta_{1},\beta_{2})$ is defined in \eqref{e5.17}. 
Then the inequalities \eqref{e5.1}, \eqref{e5.2}, \eqref{e5.13}, and Lemma \ref{l5.2} give  
\begin{align*}
\int_{0}^{r}\frac{m(t)}{t}dt
-\frac{4r}{\pi}
+(2\ell+2)\ln r 
 >-\infty, 
~~r\rightarrow\infty.
\end{align*}
The prove is completed by Theorem \ref{t4.3}. 
\end{proof}

\begin{remark}\label{r5.5}
Let $\ell=-1/2$ and $q \in L^{1}(0,1)$. If $\beta_{1}=\infty$ or $\beta_{2}=\infty$, then the result of Corollary \ref{t5.8} also holds.
If $\beta_{1}\neq\infty$ and $\beta_{2}\neq\infty$, then the result of Corollary \ref{t5.8} holds provided that an additional eigenvalue is added to the original condition. 
\end{remark}

\bigskip

\noindent {\bf Acknowledgments.}
The authors are grateful to Professor Duan Yongjiang and Professor Li Yufei for their helpful comments and suggestions, which improved and strengthened the presentation of this paper.
The author Xu was supported by the National Natural Science Foundation of China (12501216) and the Natural Science Foundation of the Jiangsu Province of China (BK20241437).

\vskip 0.5cm
\noindent {\bf Data availability.}
No data was used for the research described in the article.

\vskip 0.5cm
\noindent {\bf Conflict of Interest Statement.} 
The authors report no conflict of interest.

\end{document}